\documentclass[12pt]{amsart}

\usepackage[cp1251]{inputenc}
\usepackage{amsthm,amssymb,amsmath,amsfonts}
\usepackage{graphicx}
\usepackage[matrix,arrow,curve]{xy}
\newcommand{\CC}{\ensuremath{\mathbb{C}}}

\newcommand{\Z}{\ensuremath{\mathbb{Z}}}

\newcommand{\F}{\ensuremath{\mathbb{F}}}

\newcommand{\CG}{\ensuremath{\mathfrak{C}}}

\newcommand{\DG}{\ensuremath{\mathfrak{D}}}

\newcommand{\AG}{\ensuremath{\mathfrak{A}}}

\newcommand{\SG}{\ensuremath{\mathfrak{S}}}

\newcommand{\ka}{\ensuremath{\Bbbk}}

\newcommand{\kka}{\ensuremath{\overline{\Bbbk}}}

\newcommand{\XX}{\ensuremath{\overline{X}}}

\newcommand{\Pro}{\ensuremath{\mathbb{P}}}

\newcommand{\Aut}{\ensuremath{\operatorname{Aut}}}

\newcommand{\Gal}{\ensuremath{\operatorname{Gal}}}

\newcommand{\Pic}{\ensuremath{\operatorname{Pic}}}

\newcommand{\ord}{\ensuremath{\operatorname{ord}}}

\newcommand{\mmatr}[4]{\ensuremath{\left(
\begin{array}{ccc} #1&#2\\#3&#4\\
\end{array}
\right)}}

\makeatletter
\@addtoreset{equation}{section}
\makeatother

\newtheorem{theorem}[equation]{Theorem}
\newtheorem{proposition}[equation]{Proposition}
\newtheorem{lemma}[equation]{Lemma}
\newtheorem{corollary}[equation]{Corollary}

\theoremstyle{definition}
\newtheorem{example}[equation]{Example}
\newtheorem{definition}[equation]{Definition}

\theoremstyle{remark}
\newtheorem{remark}[equation]{Remark}

\title{Quotients of conic bundles}
\address{Institute for Information Transmission Problems, 19 Bolshoy Karetnyi side-str., Moscow 127994, Russia}
\address{Laboratory of Algebraic Geometry, GU-HSE, 7 Vavilova str., Moscow 117312, Russia}
\email{trepalin@mccme.ru}
\thanks{The author was partially supported by AG Laboratory HSE, RF government grant, ag. 11.G34.31.0023 and the grants RFFI 15-01-02164-a and N.SH.-2998.2014.1}
\author{Andrey Trepalin}

\begin{document}

\begin{abstract}
Let $\ka$ be an arbitrary field of characteristic zero. In this paper we study quotients of $\ka$-rational conic bundles over $\Pro^1_{\ka}$ by finite groups of automorphisms. We construct smooth minimal models for such quotients. We show that any quotient is birationally equivalent to a quotient of other $\ka$-rational conic bundle cyclic group $\CG_{2^k}$ of order $2^k$, dihedral group $\DG_{2^k}$ of order $2^k$, alternating group $\AG_4$ of degree $4$, symmetric group $\SG_4$ of degree $4$ or alternating group $\AG_5$ of degree $5$ effectively acting on the base of conic bundle. Also we construct infinitely many examples of such quotients which are not $\ka$-birationally equivalent to each other.

\end{abstract}

\maketitle
\section{Introduction}

Let $\ka$ be a field of characteristic zero, $K = \ka \left(x_1, \ldots x_n \right)$ be its purely transcendental extension and $G$ be a finite group acting on $\ka$. The main purpose of this paper is to study when the field of invariants $K^G$ is rational (i.e. purely transcendental) and construct some examples of non-rational fields of invariants.

In the language of algebraic geometry this problem can be formulated in the following way. Let $X$ be a $\ka$-rational variety and $G$ be a finite subgroup of $\Aut(X)$. When the quotient $X / G$ is $\ka$-rational? What is the $\ka$-birational classification of quotients $X/G$? This is a special case of a $\ka$-unirational varieties classification problem.

The answer for $n = 1$ is a classical result.

\begin{theorem}[J. L\"uroth \cite{Lur76}]
\label{Luroth}
Any unirational curve is rational.
\end{theorem}

If $n=2$ and the field $\ka$ is algebraically closed then rationality of any unirational surface follows from Castelnuovo's rationality criterion (see \cite{Cast94} or \cite[Corollary V.5]{B96}).

\begin{theorem}
\label{Castelnuovo}
Let $\ka$ be an algebraically closed field of characteristic zero. Then any $\ka$-unirational surface is $\ka$-rational.
\end{theorem}

If the field $\ka$ is not algebraically closed then the analog of Theorem~\ref{Castelnuovo} is not true. For example, any del Pezzo surface of degree $4$ is $\ka$-unirational (see \cite[Theorem 7.8]{Man74}) but not all of them are $\ka$-rational (see \cite[Chapter~4]{Isk96}). Moreover, the unirationality degree of del Pezzo surface of degree $4$ equals $1$ or $2$ thus this surface is either $\ka$-birationally trivial or $\ka$-birationally equivalent to a quotient of $\ka$-rational surface by group of order $2$.

If a variety $X$ is $\ka$-birationally equivalent to a quotient of $\ka$-rational surface by a finite group of automorphisms then $X$ is called Galois $\ka$-unirational. Now there are no known ways to prove that a given geometrically rational (i.e. rational over algebraic closure of the field~$\ka$) surface is not $\ka$-unirational. But Galois $\ka$-unirational surfaces can be studied in explicit way since groups acting on $\ka$-rational surfaces are described (see \cite{DI1}). In this paper we study Galois $\ka$-unirational surfaces which are $\ka$-birationally equivalent to quotients of conic bundles. In following papers we plan to consider quotients of del Pezzo surfaces and get full classification of Galois $\ka$-unirational surfaces.

If $X$ is a $\ka$-rational surface and $G \subset \Aut(X)$ is a finite group by applying $G$-minimal model program one can obtain a $G$-minimal surface $S$ such that the quotients $X / G$ and $S / G$ are $\ka$-birationally equivalent. Any $G$-minimal surface is either a del Pezzo surface or a conic bundle (see {\cite[Theorem 1]{Isk79}}). Thus any quotient of a $\ka$-rational surface is birationally equivalent to a quotient of a del Pezzo surface or a conic bundle.

It is important to find examples of non-$\ka$-rational quotients of conic bundles since degree of del Pezzo surfaces is bounded and for a conic bundle $S$ the number $K_S^2$ can be arbitrarily small negative integer. Also classification of quotients of conic bundles can be useful in studying $\ka$-unirationality of conic bundles.

The main results of this paper are the following.

\begin{theorem}
\label{Cbundleclasses}
Let $\ka$ be a field of characteristic zero, $X$ be a $\ka$-rational surface and $G$ be a finite group acting on $X$. Then the quotient $X / G$ is $\ka$-birationally equivalent to either a quotient of a $\ka$-rational del Pezzo surface by a finite group of automorphisms, or a quotient of $\ka$-rational conic bundle by cyclic group $\CG_{2^k}$ of order~$2^k$, dihedral group $\DG_{2^k}$ of order~$2^k$, alternating group $\AG_4$ of degree $4$, symmetric group $\SG_4$ of degree $4$ or alternating group $\AG_5$ of degree $5$.
\end{theorem}

\begin{theorem}
\label{Oddquotient}
Let $\ka$ be a field of characteristic zero, $X$ be a $\ka$-rational surface and $G$ be a finite group acting on $X$. Assume that $|G|$ is odd, $|G| > 10$ and $|G| \ne 15$. If $G$ is cyclic or $|G|$ is not divisible by $3$ then $X / G$ is $\ka$-rational.
\end{theorem}

In terms of fields of invariants Theorem \ref{Oddquotient} can be written in the following way.

\begin{theorem}
\label{Oddinvariant}
Let $\ka$ be a field of characteristic zero, $K = \ka(x , y)$ and $G$ be a finite group acting on $K$ and preserving $\ka$. Assume that $|G|$ is odd, $|G| > 10$ and $|G| \ne 15$. If $G$ is cyclic or $|G|$ is not divisible by $3$ then $K^G$ is purely transcendental extension of $\ka$.
\end{theorem}

The next proposition generalizes particular case of \cite[Chapter~4]{Isk96} (see also Theorem \ref{ratcrit} below) for nontrivial groups $G$.

\begin{proposition}
\label{RCbundle}
If $X$ is a $\ka$-rational $G$-equivariant conic bundle and $K_X^2 \geqslant 5$ then $X / G$ is $\ka$-rational.
\end{proposition}

We need the following definition (see \cite{Al94}) to formulate the main result of this paper.

\begin{definition}
\label{birunboundness}
A class $\mathfrak{V}$ of varieties is \textit{$\ka$-birationally bounded} if there is a morphism $\varphi: \mathcal{X} \rightarrow \mathcal{S}$ between algebraic schemes of finite type such that every member of $\mathfrak{V}$ is $\ka$-birationally equivalent to one of the geometric fibres of $\varphi$. We say that $\mathfrak{V}$ is \textit{$\ka$-birationally unbounded} if it is not $\ka$-birationally bounded.
\end{definition}

There are many examples of birationally unbounded classes of varieties of dimension $3$ or greater over $\CC$. In particular, for $n \geqslant 3$ smooth $n$-dimensional varieties admitting a conic bundle structure are birationally unbounded (see \cite[Theorem 1.5]{Ok10}).

We show that for some field $\ka$ Galois $\ka$-unirational surfaces are $\ka$-birationally unbounded. Quotients of del Pezzo surfaces are $\ka$-birationally bounded since $G$-minimal del Pezzo surfaces are $\ka$-birationally bounded. Therefore we consider quotients of conic bundles.

\begin{theorem}
\label{Unboundness}
Let $\ka$ be a field of characteristic zero such that not all elements of $\ka$ are squares and $G$ be a finite group of automorphisms of $\Pro^1_{\ka}$. Then the class of $G$-quotients of $\ka$-rational conic bundles is $\ka$-birationally unbounded in the following cases:

\begin{itemize}

\item If $\ka$ contains $\xi_k = e^{\frac{2\pi i}{k}}$ and $G$ is a cyclic group $\CG_{2k}$ of order $2k$;

\item If $\ka$ contains $\cos \frac{2\pi}{k}$ and $G$ is a dihedral group $\DG_{2k}$ of order $2k$;

\item If $\ka$ contains $i$ and $G$ is an alternating group $\AG_4$ of degree $4$;

\item If $\ka$ contains $i$ or $i\sqrt{2}$ and $G$ is a symmetric group $\SG_4$ of degree~$4$;

\item If $\ka$ contains $i$ and $\sqrt{5}$ and $G$ is an alternating group $\AG_5$ of degree $5$.

\end{itemize}

\end{theorem}

\begin{remark}
\label{Unfields}
Note that quotients of $\ka$-rational surfaces by a group $G$ up to $\ka$-birational equivalence bijectively correspond to fields of invariants $\ka(x,y)^G$ up to isomorphism.
\end{remark}

\begin{corollary}
\label{Uncremona}
In the assumptions of Theorem \ref{Unboundness} there are infinitely many conjugacy classes of subgroups isomorphic to $G$ in $\operatorname{Bir} \Pro^2_{\ka}$.
\end{corollary}

The plan of this paper as follows.

In Section 2 we review main notions of minimal model program, facts about conic bundles and rationality.

In Section 3 we review some facts about finite subgroups of $\mathrm{PGL}_2 \left( \ka \right)$ and fixed points of these groups on $\Pro^1_{\ka}$.

In Section 4 we study how the geometric group $G$ acts on a \mbox{$G$-minimal} conic bundle $X$. As a by-product we prove Proposition~\ref{RCbundle}.

In Section 5 we study the action of the Galois group $\Gal\left( \kka / \ka \right)$ on singular fibres of conic bundles and connection of this action with the action of the geometric group $G$. Also we prove Theorem \ref{Cbundleclasses} and Theorem \ref{Oddquotient}.

In Section 6 we construct examples of non-$\ka$-rational quotients for all groups listed in Theorem \ref{Unboundness} and prove Corollary~\ref{Uncremona}.

The author is grateful to his adviser Yu.\,G.\,Prokhorov and to C.\,A.\,Shramov for useful discussions. Also the author would like to thank the reviewers of this paper for many useful comments.

We use the following notation.

\begin{itemize}

\item $\ka$ denotes an arbitrary field of characteristic zero,

\item $\kka$ denotes the algebraic closure of a field $\ka$,

\item $\XX = X \otimes \kka$,

\item $|G|$ denotes the order of a group $G$,

\item $\ord g$ denotes the order of an element $g \in G$,

\item $\CG_n$ denotes the cyclic group of order $n$,

\item $\mathfrak{D}_{2n}$ denotes the dihedral group of order $2n$, in particular \mbox{$\DG_4 \cong \CG_2^2$},

\item $\mathfrak{S}_n$ denotes the symmetic group of degree $n$,

\item $\mathfrak{A}_n$ denotes the alternating group of degree $n$,

\item $(i_1 i_2 \ldots i_j)$ denotes a cyclic permutation of $i_1$, \ldots, $i_j$,

\item $\langle g_1, \ldots , g_k \rangle$ denotes the group generated by $g_1$, \ldots, $g_k$,

\item $\operatorname{diag}(a, b) = \mmatr{a}{0}{0}{b}$,

\item $K_X$ denotes the canonical divisor of a variety $X$,

\item $\F_n$ denotes the rational ruled (Hirzebruch) surface $\Pro_{\Pro^1}(\mathcal{O}~\oplus~\mathcal{O}(n))$,

\item $X \approx Y$ means that $X$ and $Y$ are $\ka$-birationally equivalent,

\item $T_p X$ denotes the Zariski tangent space of $X$ at a point $p$,

\item $\xi_k = e^{\frac{2\pi}{k}}$.

\end{itemize}

\section{Preliminaries}

In this section we review main notions and results of the $G$-equivariant minimal model program following the papers \cite{Man67}, \cite{Isk79}, \cite{DI1}, \cite{DI2}. Throughout this section $G$ is a finite group.

\begin{definition}
\label{rationality}
A {\it $\ka$-rational surface} $X$ is a surface over $\ka$ such that $X$ is birationally equivalent to $\Pro^2_{\ka}$.

We say that $X$ is \textit{rational} if $\XX = X \otimes \kka$ is $\kka$-rational.

A surface $X$ over $\ka$ is a {\it $\ka$-unirational surface} if there exists a $\ka$-rational variety $Y$ and a dominant rational map $\varphi: Y \dashrightarrow X$.
\end{definition}

\begin{definition}
\label{minimality}
A {\it $G$-surface} is a pair $(X, G)$ where $X$ is a projective surface over $\ka$ and $G$ is a subgroup of $\Aut_{\ka}(X)$. A morphism of $G$-surfaces $f: X \rightarrow X'$ is called a {\it $G$-morphism $(X, G) \rightarrow (X', G)$} if for each $g \in G$ one has $fg = gf$.

A smooth $G$-surface $(X, G)$ is called {\it $G$-minimal} if any birational morphism of smooth $G$-surfaces $(X, G) \rightarrow (X',G)$ is an isomorphism.

Let $(X, G)$ be a smooth $G$-surface. A $G$-minimal surface $(Y, G)$ is called a {\it minimal model} of $(X, G)$ if there exists a birational $G$-morphism $X \rightarrow Y$.
\end{definition}

The classification of $G$-minimal rational surfaces is well-known due to V.\,Iskovskikh and Yu.\,Manin (see \cite{Isk79} and \cite{Man67}). We recall some important notions before surveying it.

\begin{definition}
\label{Cbundledef}
A smooth rational $G$-surface $(X, G)$ admits a {\it conic bundle} structure if there exists a $G$-morphism $\varphi: X \rightarrow B$ such that any scheme fibre is isomorphic to a reduced conic in~$\Pro^2_{\ka}$ and $B$ is a smooth curve. The curve $B$ is called the {\it base} of conic bundle.

\end{definition}

Let $\overline{\varphi}: \XX \rightarrow \overline{B}$ be a conic bundle. A general fibre of $\overline{\varphi}$ is isomorphic to $\Pro^1_{\kka}$. The fibration $\overline{\varphi}$ has a finite number of singular fibres which are degenerate conics. Any irreducible component of a singular fibre is a $(-1)$-curve. If $n$ is the number of singular fibres of $\overline{\varphi}$ then $K_X^2 + n = 8$.

\begin{definition}
\label{relmin}
Let $X$ be a $G$-surface that admits a conic bundle structure \mbox{$\varphi: X \rightarrow B$}. The conic bundle is called \textit{relatively $G$-minimal} over $B$ if for any decomposition of $\varphi$ into $G$-morphisms
$$
X \xrightarrow{\psi} X' \rightarrow B
$$
\noindent such that the morphism $\psi$ is birational the morphism $\psi$ is isomorphism.
\end{definition}

A conic bundle \mbox{$\varphi: X \rightarrow B$} is relatively $G$-minimal over $B$ if and only if $\operatorname{rk} \Pic(X)^G = 2$.

\begin{definition}
\label{DPdef}
A {\it del Pezzo surface} is a smooth projective surface $X$ such that the anticanonical divisor $-K_X$ is ample.
\end{definition}

\begin{theorem}[{\cite[Theorem 1]{Isk79}}]
\label{Minclass}
Let $X$ be a $G$-minimal rational $G$-surface. Then either $X$ admits a conic bundle structure over a rational curve with $\Pic(X)^{G} \cong \Z^2$, or $X$ is a del Pezzo surface with \mbox{$\Pic(X)^{G} \cong \Z$}.
\end{theorem}

\begin{theorem}[{cf. \cite[Theorem 4]{Isk79}, \cite[Theorem 5]{Isk79}}]
\label{MinCB}
Let $X$ admit a $G$-equivariant structure of a conic bundle. Then:

(i) If $K_X^2 = 3, 5, 6, 7$ or $X \cong \F_1$ then $X$ is not $G$-minimal.

(ii) If $K_X^2 = 8$ then $X$ is isomorphic to $\F_n$.

(iii) If $K_X^2 \ne 3, 5, 6, 7$, $X \ne \F_1$ and $X$ is relatively $G$-minimal then $X$ is $G$-minimal.

(iv) If $K_X^2 = 3, 5, 6$ and $X$ is relatively $G$-minimal, then $X$ is a del Pezzo surface.

(v) If $K_X^2 = 7$, then $X$ is not relatively $G$-minimal.

\end{theorem}

The following lemma is implied by \cite[Theorem 1.6]{Isk67}. For convenience of reader we give a sketch of proof.

\begin{lemma}
\label{rigidbundle}
Let $\varphi_1: X_1 \rightarrow B_1$ and $\varphi_2: X_2 \rightarrow B_2$ be two relatively minimal conic bundles such that $\varphi_1$ has $m$ singular fibres over points $p_1$, \ldots, $p_m$ in $\overline{B}_1$ and $\varphi_2$ has $l$ singular fibres over points $q_1$, \ldots, $q_l$ in~$\overline{B}_2$. If $X_1$ and $X_2$ are birationally equivalent and $K_{X_1}^2 \leqslant 0$ then $m = l$ and there exists an isomorphism $B_1 \rightarrow B_2$ such that the points $p_1$, \ldots, $p_m$ map to the points $q_1$, \ldots, $q_l$ up to permutation.
\end{lemma}

\begin{proof}

Let $f: X_1 \dashrightarrow X_2$ be a birational map. Then by \cite[Theorem 1.6]{Isk67} there exists an isomorphism $u: B_1 \rightarrow B_2$ such that $u\varphi_1 = \varphi_2f$. Therefore the map $f$ decomposes into Sarkisov links of type $II$. The image of a singular fibre under such link is a singular fibre. Thus $m = l$ and for any $p_i$ one has $u(p_i) = q_j$.

\end{proof}

The following theorem is an important criterion of $\ka$-rationality over an arbitrary perfect field $\ka$.

\begin{theorem}[{\cite[Chapter 4]{Isk96}}]
\label{ratcrit}
A minimal rational surface $X$ over a perfect field $\ka$ is $\ka$-rational if and only if the following two conditions are satisfied:

(i) $X(\ka) \ne \varnothing$;

(ii) $K_X^2 \geqslant 5$.
\end{theorem}

\begin{corollary}
\label{Cbundleratcrit}
Let $X$ admit a relatively minimal conic bundle structure over a perfect field $\ka$. The surface $X$ is $\ka$-rational if and only if the following two condition are satisfied:

(i) $X(\ka) \ne \varnothing$;

(ii) $K_X^2 \geqslant 5$.
\end{corollary}
\begin{proof}
If $K_X^2 \geqslant 5$ and $X(\ka) \ne \varnothing$ then any minimal model $Y$ of $X$ is $\ka$-rational by Theorem \ref{ratcrit}. Thus $X$ is $\ka$-rational.

If $K_X^2 < 5$ and $X$ is relatively minimal then $X$ is minimal if $K_X^2 \ne 3$ by Theorem \ref{MinCB} (iii). Therefore if $K_X^2 \ne 3$ then $X$ is not $\ka$-rational by Theorem \ref{ratcrit}. If $K_X^2 = 3$ then there exists a minimal model $Y$ of $X$ such that $K_Y^2 = 4$ (see the proof of \cite[Theorem 4]{Isk79}). Thus $Y$ and $X \approx Y$ are not $\ka$-rational by Theorem \ref{ratcrit}.
\end{proof}

For convenience of the reader we use the following definitions.

\begin{definition}
\label{MMPred}
Let $X$ be a $G$-surface, $\widetilde{X} \rightarrow X$ be the minimal resolution of singularities, and $Y$ be a $G$-equivariant minimal model of~$\widetilde{X}$. We call the surface $Y$ a \textit{$G$-MMP-reduction} of $X$.
\end{definition}

\begin{definition}
\label{MMPrelred}
Let $X$ be a $G$-surface and let $\varphi: X \rightarrow B \cong \Pro^1_{\ka}$ be a $G$-morphism such that its general fibre is a smooth rational curve. Let $\widetilde{X} \rightarrow X$ be its ($G$-equivariant) minimal resolution of singularities, and $Y$ be a $G$-equivariant relatively minimal model of $\widetilde{X}$. We call the surface $Y$ a \textit{relative $G$-MMP-reduction} of $X$ over $B$.
\end{definition}

\section{Automorphisms of $\Pro^1_{\ka}$}

In this section we review some well-known facts about finite subgroups of $\mathrm{PGL}_2\left( \ka \right)$, i.e. group of automorphisms of $\Pro^1_{\ka}$.

\begin{theorem}[{see \cite[Chapter III]{Bl17}}]
\label{PGL2class}
Let $G$ be a finite subgroup of $\mathrm{PGL}_2 \left( \kka \right)$. Then $G$ is one of the following group: $\CG_k$, $\DG_{2k}$, $\AG_4$, $\SG_4$, $\AG_5$.
\end{theorem}

One has $\mathrm{PGL}_2 \left( \ka \right) \subset \mathrm{PGL}_2 \left( \kka \right)$ therefore any finite subgroup of $\mathrm{PGL}_2 \left( \ka \right)$ is one of the groups listed in Theorem \ref{PGL2class}.

We need some facts about fixed points of elements of finite order.

\begin{lemma}
\label{fixedpoints}
Elements $g_1, g_2 \in \mathrm{PGL}_2 \left( \kka \right)$ such that the group $H = \langle g_1, g_2\rangle$ is finite have the same pair of fixed points on $\Pro^1_{\kka}$ if and only if the group $H$ is cyclic. If two elements $g_1$, $g_2$ of finite subgroup of $\mathrm{PGL}_2 \left( \kka \right)$ have a common fixed point on $\Pro^1_{\kka}$ then they have the same pair of fixed points.
\end{lemma}
\begin{proof}
Let two elements $g_1, g_2 \in \mathrm{PGL}_2 \left( \kka \right)$ have a common fixed point \mbox{$p \in \Pro^1_{\kka}$} and the group $H = \langle g_1, g_2\rangle$ is finite. The natural representation
$$
H \rightarrow \mathrm{GL} \left(T_p \Pro^1_{\kka} \right) \cong \kka^*
$$
\noindent is faithfull (see e.g. \cite[Lemma 4]{Pop14}). Thus $H$ is cyclic and all elements of $H$ have common pair of fixed points.
\end{proof}

\begin{lemma}[{cf. \cite[4.4.6]{Spr77}}]
\label{orbits}
Let a finite group $G$ faithfully act on~$\Pro^1_{\kka}$. Then orbits of points with non-trivial stabilizer have the following lengths:
\begin{itemize}
\item If $G \cong \CG_k$, $k \geqslant 2$, then there are two orbits of length $1$.
\item If $G \cong \DG_{2k}$, $k \geqslant 2$, then there are an orbit of length $2$ and two orbits of length $k$.
\item If $G \cong \AG_4$, then there are two orbits of length $4$ and an orbit of length $6$.
\item If $G \cong \SG_4$, then there are an orbit of length $6$, an orbit of length $8$ and an orbit of length $12$.
\item If $G \cong \AG_5$, then there are an orbit of length $12$, an orbit of length $20$ and an orbit of length $30$.
\end{itemize}
\end{lemma}

Now we want to construct some explicit representation of these groups.

We use the following notation:
$$
\mathrm{Id} = \mmatr{1}{0}{0}{1}, \quad I = \mmatr{-i}{0}{0}{i}, \quad J = \mmatr{0}{-1}{1}{0}, \quad K = \mmatr{0}{i}{i}{0},
$$
$$
R_k = \mmatr{\xi_k}{0}{0}{1}, \quad S = \mmatr{0}{1}{1}{0}, \quad \varphi = \frac{1 + \sqrt{5}}{2}.
$$

\begin{lemma}
\label{Fixeddefined}
Let $g \in \mathrm{PGL}_2 \left( \ka \right)$ be an element of finite order \mbox{$\ord g = k > 2$}. Then fixed points on $\Pro^1_{\ka}$ of the element $g$ are defined over $\ka$ if and only if the field $\ka$ contains an element $\xi_k$.
\end{lemma}

\begin{proof}
Assume that the element $g$ is represented in $\mathrm{PGL}_2 \left( \ka \right)$ as an operator $\widetilde{R}_k$. Then over $\kka$ the operator $\widetilde{R}_k$ is conjugate to $\lambda R_k$, $\lambda \in \kka$. One has
$$
\operatorname{tr} \widetilde{R}_k = \lambda \left( \xi_k + 1 \right) \in \ka \qquad \text{and} \qquad \operatorname{det} \widetilde{R}_k = \lambda^2 \xi_k \in \ka.
$$
Note that fixed points of $\widetilde{R}_k$ are defined over $\ka$ if and only if eigenvalues $\lambda \xi_k$ and $\lambda$ are defined over $\ka$. If $\xi_k \in \ka$ and $k \ne 2$ then
$$
\lambda = \frac{\operatorname{tr} \widetilde{R}_k}{\xi_k + 1}
$$
\noindent and $\lambda \xi_{k}$ are defined over $\ka$. If the eigenvalues $\lambda \xi_k \in \ka$ and $\lambda \in \ka$ then $\xi_k \in \ka$.
\end{proof}

\begin{lemma}
\label{Crepr}
Let a field $\ka$ contain element $\xi_k$. Then there is a cyclic subgroup $\CG_k \subset \mathrm{PGL}_2\left( \ka \right)$ such that the fixed points of its generator on $\Pro^1_{\ka}$ are defined over $\ka$.
\end{lemma}

\begin{proof}
The element $R_k$ is defined over $\ka$ and generates a group $\CG_k$. The fixed points of $R_k$ on $\Pro^1_{\ka}$ are $(1 : 0)$ and $(0 : 1)$. These points are defined over $\ka$.
\end{proof}

\begin{lemma}
\label{Drepr}
Let a field $\ka$ contain element $\cos \frac{2\pi}{k}$. Then there is a subgroup $\DG_{2k} \subset \mathrm{PGL}_2\left( \ka \right)$. In this group there is an element of order $2$ such that its fixed points on $\Pro^1_{\ka}$ are defined over $\ka$.
\end{lemma}

\begin{proof}
If a field $\ka$ contains $\cos \frac{2\pi}{k}$ then $\ka$ contains $\cos^2 \frac{\pi}{k}$. Consider an element
$$
\widetilde{R}_k = \mmatr{4\cos^2 \frac{\pi}{k} - 1}{-1}{1}{1}.
$$
\noindent One has $\ord \widetilde{R}_k = k$. Moreover, $S \widetilde{R}_k S^{-1} = \widetilde{R}_k^{-1}$. Therefore elements $\widetilde{R}_k$ and $S$ generate a group $\DG_{2k}$. The fixed points of $S$ on $\Pro^1_{\ka}$ are $(1 : 1)$ and $(-1 : 1)$. These points are defined over $\ka$.
\end{proof}

\begin{remark}
\label{Drepr1}
Note that if a field $\ka$ does not contain element $\cos \frac{2\pi}{k}$ then there are no elements of order $k$ in $\mathrm{PGL}_2\left( \ka \right)$.
\end{remark}

\begin{lemma}
\label{A4repr}
Let a field $\ka$ contain element $i$. Then there is a subgroup $\AG_{4} \subset \mathrm{PGL}_2\left( \ka \right)$. In this group there is an element of order $2$ such that its fixed points on $\Pro^1_{\ka}$ are defined over $\ka$.
\end{lemma}

\begin{proof}
Note that elements $I$, $J$ and $K$ are defined over $\ka$. The elements $I$ and $\mathrm{Id} + I + J + K$ generate the group $\AG_4$. The fixed points of $I$ on $\Pro^1_{\ka}$ are $(1 : 0)$ and $(0 : 1)$. These points are defined over $\ka$.
\end{proof}

\begin{lemma}
\label{S4repr}
Let a field $\ka$ contain element $i$ or $i\sqrt{2}$. Then there is a subgroup $\SG_{4} \subset \mathrm{PGL}_2\left( \ka \right)$. In this group there is an element of order $2$ such that its fixed points on $\Pro^1_{\ka}$ are defined over $\ka$.
\end{lemma}

\begin{proof}
If a field $\ka$ contains element $i$ then elements $I$, $J$ and $K$ are defined over $\ka$. The elements $I$, $\mathrm{Id} + I + J + K$ and $I + J$ generate the group $\SG_4$. The fixed points of $I$ on $\Pro^1_{\ka}$ are $(1 : 0)$ and $(0 : 1)$. These points are defined over $\ka$.

Let a field $\ka$ contain element $i\sqrt{2}$. Consider an element
$$
\widetilde{I} =  \sqrt{2} I - iK = \mmatr{-i\sqrt{2}}{1}{1}{i\sqrt{2}}.
$$
\noindent The elements $\widetilde{I}$, $\mathrm{Id} + \widetilde{I} + J + \widetilde{I}J$ and
$$
\widetilde{I} + J = \mmatr{-i\sqrt{2}}{0}{2}{i\sqrt{2}}
$$
\noindent generate the group $\SG_4$. The fixed points of $\widetilde{I} + J$ are $(0 : 1)$ and $(-i\sqrt{2} : 1)$. These points are defined over $\ka$.
\end{proof}

\begin{lemma}
\label{A5repr}
Let a field $\ka$ contain elements $i$ and $\varphi$. Then there is a subgroup $\AG_{5} \subset \mathrm{PGL}_2\left( \ka \right)$. In this group there is an element of order $2$ such that its fixed points on $\Pro^1_{\ka}$ are defined over $\ka$.
\end{lemma}

\begin{proof}
Note that elements $I$, $J$ and $K$ are defined over $\ka$. The elements $I$, $\mathrm{Id} + I + J + K$ and $\mathrm{Id} + \varphi I + \varphi^{-1} J$ generate the group $\AG_5$. The fixed points of $I$ on $\Pro^1_{\ka}$ are $(1 : 0)$ and $(0 : 1)$. These points are defined over~$\ka$.
\end{proof}

\section{Geometry of fibres and sections}

In this section we prove the following theorem.

\begin{theorem}
\label{Cbundle}
Let $X$ be a $G$-surface that admits a $G$-equivariant conic bundle structure \mbox{$\varphi: X \rightarrow B$}. Then any relative MMP-reduction $Y$ over $B / G$ of the quotient $X / G$ admits a conic bundle structure. Denote by $G_B$ the image of $G$ under the natural map $\Aut(X) \rightarrow \Aut(B)$. Then

(i) If $K_X^2 \geqslant 5$ then $K_Y^2 \geqslant 5$;

(ii) If $K_X^2 < 5$ then $K_Y^2 \geqslant K_X^2$. If furthermore $K_X^2 = K_Y^2$ then $K_X^2 = K_Y^2 = 4$ and $G_B \cong \CG_2$ or $G_B \cong \DG_4$.

\end{theorem}

\begin{remark}
\label{Galoisunirat}
Any del Pezzo surface $X$ with $K_X^2 \leqslant 3$ and $\rho(X) = 1$ is not birationally equivalent to a conic bundle (see \cite[Theorem 5.7]{Man67}). Therefore by Theorem \ref{Cbundle} such a surface is not birationally equivalent to any quotient of any conic bundle.
\end{remark}

Throughout this section we use the following notation.

Let a finite group $G$ faithfully act on a relatively $G$-minimal conic bundle \mbox{$\varphi: X \rightarrow B \cong \Pro^1_{\ka}$} and $n$ be the number of singular fibres of $\varphi$ over $\kka$. The morphism $\varphi$ is $G$-equivariant. It means that there exists an exact sequence:
$$
1 \rightarrow G_F \hookrightarrow G \twoheadrightarrow G_B \rightarrow 1
$$
where $G_F$ is a group of automorphisms of the generic fibre, acting trivially on the base $B \cong \Pro^1_{\ka}$, and $G_B$ is a group of automorphisms of~$B$.

Note that the surface $X / G_F$ admits a natural structure of $G_B$-equivariant bundle $X / G_F \rightarrow B$.

\begin{lemma}
\label{fibrequotient}
General fibre of the bundle $X / G_F \rightarrow B$ is a smooth conic. Let $Y \rightarrow B$ be a relative $G_B$-MMP-reduction of the bundle $X / G_F \rightarrow B$. The number of singular fibres of $\overline{Y} \rightarrow \overline{B}$ is less or equal to $n$.
\end{lemma}

\begin{proof}
Let us consider a smooth fibre $F \cong \Pro^1_{\kka}$ of the conic bundle $\XX \rightarrow \overline{B}$. The group $G_F$ acts on $F$ and the quotient $F / G_F$ is isomorphic to $\Pro^1_{\kka}$. Each non-trivial automorphism $g \in G_F$ has exactly two fixed points on $F$, but in a neighbourhood of these points $g$ acts as a reflection because it acts trivially on $B$. Therefore there are no singular points lying on $F / G_F$ in $\XX / G_F$.

The $G$-morphism $\varphi: X \rightarrow B$ induces a $G_B$-morphism \mbox{$\psi: X / G_F \rightarrow B$}. A general fibre of $\psi$ over $\kka$ is isomorphic to $\Pro^1_{\kka}$. Thus by applying $G_B$-equivariant Minimal Model Program over $B$ to $X / G_F$ one can obtain that a relative $G_B$-MMP-reduction $Y$ of $X / G_F$ is a conic bundle. If the fibre over $p \in \overline{B}$ of the bundle $\XX \rightarrow \overline{B}$ is smooth then the fibre over $p$ of the bundle $\overline{Y} \rightarrow \overline{B}$ is also smooth. Thus the number of singular fibres of $\overline{Y} \rightarrow \overline{B}$ is less or equal to $n$.
\end{proof}

From now on we assume that $G_F$ is trivial and $G = G_B$. If $X$ is minimal then the components of each singular fibre of the conic bundle $\XX \rightarrow \overline{B}$ are permuted by some element of $G \times \Gal\left(\kka / \ka\right)$ since otherwise we can contract $G \times \Gal\left(\kka / \ka\right)$-invariant set of $(-1)$-curves lying in different fibers of the conic bundle.

\begin{lemma}
\label{permutation}
Any element $g \in G$ does not permute the components of any singular fibre.
\end{lemma}

\begin{proof}
Assume that there exists an element $g \in G$ permuting the components of a singular fibre $F$. Obviously, $\ord g$ is even. The element $g$ faithfully acts on the base $B \cong \Pro^1_{\ka}$ therefore there are exactly two fixed points $p_1$ and $p_2$ on $\overline{B}$. If there exists a singular fibre $F$ over any other point $p$ and an integer $k$ such that $g^k$ permutes components of this fibre then $g^k$ has at least three fixed points $p$, $p_1$ and $p_2$ on the base. Thus the action of $g^k$ is trivial on the base and non-trivial on $X$. This contradicts the assumption that $G_F$ is trivial.

Apply $\langle g \rangle$-minimal model program to the surface $\XX$. The obtained surface $\overline{Y}$ has $1$ or $2$ singular fibres over points $p_1$ and $p_2$. Thus $K_{\overline{Y}}^2$ equals $6$ or $7$. By Theorem \ref{MinCB} (iv) the surface $\overline{Y}$ is a del Pezzo surface of degree $6$.

Denote six $(-1)$-curves on $\overline{Y}$ by $E_1$, $E_2$, $\ldots$, $E_6$, where $E_i^2 = -1$ and $E_i \cdot E_j = 1$ if $i = j \pm 1 \pmod 6$, otherwise $E_i \cdot E_j = 0$. The element $g$ acts on these curves as follows:
$$
gE_1 = E_2, gE_2 = E_1, gE_3 = E_6, gE_4 = E_5, gE_5 = E_4, gE_6 = E_3.
$$ 
The curves $E_3$ and $E_6$ are sections of the conic bundle $\overline{Y} \rightarrow \overline{B}$, while $E_1 \cup E_2$ and $E_4 \cup E_5$ are singular fibres of this bundle.

Consider the element $h = g^{\frac{\ord g}{2}}$. Note that $h$ acts faithfully on the base $B$ therefore all $h$-fixed points on $\overline{Y}$ lie on the singular fibres of the conic bundle $\overline{Y} \rightarrow \overline{B}$. The curves $E_i$ are not $h$-fixed since $\overline{Y}$ is relatively $\langle g \rangle$-minimal. Thus by the Lefschetz fixed-points formula the element $h$ has exactly four fixed points on $\overline{Y}$. But either all curves $E_i$ are $h$-invariant or $hE_1 = E_2$ and $hE_4 = E_5$. In the former case $h$ has six fixed points, and in the latter case $h$ has two fixed points. This contradiction finishes the proof.
\end{proof}

One has 
$$
G = G_B \subset \Aut\left( B \right) \cong \mathrm{PGL}_2 \left( \ka \right) \subset \mathrm{PGL}_2 \left( \kka \right).
$$
\noindent Therefore $G$ is one of the following groups: cyclic group $\CG_k$, dihedral group $\DG_{2k}$ (including \mbox{$\DG_4 = \CG_2^2$}), alternating group $\AG_4$, symmetric group $\SG_4$ or alternating group $\AG_5$ by Theorem \ref{PGL2class}.

We use the following well-known lemma whose proof is a direct computation.

\begin{lemma}
\label{blowup}
Let $p$ be a smooth point of a surface $S$ and $g$ be an element of $\Aut(S)$ that fixes the point $p$. Let $g$ act on $T_p S$ as $\operatorname{diag}\left(\lambda, \mu \right)$ and $\pi: \widetilde{S} \rightarrow S$ be a blowup of $S$ at the point $p$. Then $g$ has exactly two fixed points $p_1$ and $p_2$ on the exceptional divisor of $\pi$ and the actions on $T_{p_1}\widetilde{S}$ and $T_{p_2}\widetilde{S}$ have the form $\operatorname{diag}\left(\frac{\lambda}{\mu}, \mu \right)$ and $\operatorname{diag}\left(\lambda, \frac{\mu}{\lambda} \right)$ respectively.
\end{lemma}

\begin{lemma}
\label{eveninvariantfibre}
Let $g \in G$ be an element of even order. Then $g$-invariant fibres of the conic bundle $\XX \rightarrow \overline{B}$ are smooth.
\end{lemma}
\begin{proof}
Let $p$ be a $g$-fixed point on $\overline{B}$ and $F$ be a fibre over this point. We can assume that in the neighbourhood of $p$ the element $g$ acts on $\overline{B}$ as a multiplication by $\xi_k$.

Let $F$ be a singular fibre so that $F = E_1 \cup E_2$. These curves $E_1$ and $E_2$ are $g$-invariant by Lemma \ref{permutation}. Let $q_1 \in E_1$ and $q_2 \in E_2$ be $g$-fixed points different from the point $E_1 \cap E_2$.

In the neighbourhood of $q_1$ the element $g$ acts on $E_1$ as a multiplication by~$\xi_k^a$ for some $a$. If $f: \XX \rightarrow \overline{S}$ is $g$-equivariant contraction of $E_2$ then the element $g$ acts on $T_{f(E_2)} \overline{S}$ as $\operatorname{diag} \left( \xi_k, \xi_k^{-a} \right)$. Thus in the neighbourhood of $q_2$ the element $g$ acts on $E_2$ as a multiplication by $\xi_k^{-a-1}$ by Lemma \ref{blowup}. The point $p$ is fixed only by elements of cyclic groups containing $g$ by Lemma \ref{fixedpoints}. Therefore if an element $t$ of $G \times \Gal(\kka / \ka)$ permutes $E_1$ and $E_2$ then $tq_1 = q_2$, $tq_2 = q_1$ and the groups $G_{q_1} = \langle \operatorname{diag}\left( \xi_k, \xi_k^a \right) \rangle$ and $G_{q_2} = \langle \operatorname{diag}\left( \xi_k, \xi_k^{-a-1} \right) \rangle$ coincide in $\mathrm{GL}_2(\kka)$. It holds if and only if $k$ is odd and $a = \frac{k-1}{2}$. The latter equality contradicts assumption that $k$ is even.
\end{proof}

\begin{lemma}
\label{quotientoffibre}
Let $Y$ be a relative MMP-reduction of the quotient $X / G$ and $f: X \dashrightarrow Y$ be the corresponding rational map. Let $g \in G$ be an element of odd order such that the group $\langle g\rangle$ is not contained in a bigger cyclic subgroup of $G$. Then for any $g$-invariant smooth fibre $F$ of $\XX \rightarrow \overline B$ the fibre $f(F)$ of $\overline{Y} \rightarrow \overline{B} / G$ is smooth.
\end{lemma}
\begin{proof}
Let $p$ be a $g$-fixed point on $\overline{B}$ and $F$ be a fibre over this point. We can assume that in the neighbourhood of $p$ the element $g$ acts on $\overline{B}$ as a multiplication by $\xi_k$.

Let $F$ be a smooth fibre. If $g$ acts trivially on $F$ then $f(F)$ is a smooth fibre. Otherwise $g$ has two fixed points $q_1$ and $q_2$ on $F$. In the neighbourhood of $q_1$ the element $g$ acts on $F$ as a multiplication by $\xi_k^a$ for some $a$ and in the neighbourhood of $q_2$ the element $g$ acts on $F$ as a multiplication by $\xi_k^{-a}$.

Let us consider the quotient map $\pi: X \rightarrow X / G$. There are two singular points $\pi(q_1)$ and $\pi(q_2)$ on $\pi(F)$. These singularities are toric thus for their minimal resolutions their preimages are chains of rational curves. The selfintersection numbers of these curves $-s_1$, $\ldots$, $-s_i$ and $-r_1$, $\ldots$, $-r_j$ are determined by continued fractions
$$
\frac{k}{a} = s_1 - \frac{1}{s_2 - \frac{1}{\ddots - \frac{1}{s_i}}}, \qquad \frac{k}{k-a} = r_1 - \frac{1}{r_2 - \frac{1}{\ddots - \frac{1}{r_j}}}.
$$
Let us consider a minimal resolution
$$
\mu: \widetilde{X / G} \rightarrow X / G.
$$
\noindent The preimage $\mu^{-1} \left( \pi(F) \right)$ is a chain of rational curves consisting of $\mu^{-1} \left( \pi(q_1) \right)$, $\mu^{-1} \left( \pi(q_2) \right)$ and the proper transform $\mu^{-1}_*\pi(F)$. The dual graph $\Xi$ of $\mu^{-1}\left(\pi(F)\right)$ admits at most one nontrivial automorphism which switches the ends of this chain.

If there is no non-trivial automorphisms of the graph $\Xi$ then the Galois group $\Gal \left( \kka / \ka \right)$ cannot permute irreducible components of $\mu^{-1} \left( \pi(F) \right)$ and applying the relative minimal model program over $\ka$ one can obtain that $f(F)$ is a smooth fibre since for singular fibre there are not any automorphisms permutting its components. If there is a non-trivial automorphism then
$$
i = j, \quad s_1 = r_1, \quad \ldots, \quad s_i = r_j.
$$
\noindent Therefore
$$
\frac{k}{a} = \frac{k}{k-a}, \quad \text{so that} \quad a = \frac{k}{2}
$$
\noindent and the number $k$ is even. It contradicts assumption that $k$ is odd.
\end{proof}

\begin{lemma}
\label{quotientofoddfibre}
Let $\pi: X \rightarrow X / G$ be the quotient morphism, $Y$ be a relative MMP-reduction of the quotient $X / G$ and $f: X \dashrightarrow Y$ be the corresponding rational map. Let $g \in G$ be an element of odd order such that the group $\langle g\rangle$ is not contained in a bigger cyclic subgroup of $G$. Consider a $g$-invariant fibre $F$ of $\XX \rightarrow \overline{B}$. If $F$ is singular then $f(F)$ is singular fibre.
\end{lemma}

\begin{proof}
Let $p$ be a $g$-fixed point on $\overline{B}$ and $F$ be a fibre over this point. We can assume that in the neighbourhood of $p$ the element $g$ acts on $\overline{B}$ as a multiplication by $\xi_k$.

The conic bundle $X \rightarrow B$ is relatively $G$-minimal thus the components of $\pi(F)$ are permuted by the Galois group $\Gal \left( \kka / \ka \right)$.

Let $F = E_1 \cup E_2$ be a singular fibre. The element $g$ cannot act trivially on the both components of $F$. If $g$ acts trivially on a component of $F$ then on the other component there is an isolated fixed point of~$g$. Therefore there is an unique singular point lying on a component of $\pi(F)$ and components of $\pi(F)$ cannot be permuted by the Galois group $\Gal \left( \kka / \ka \right)$. Thus the element $g$ acts effectively on both components $E_1$ and $E_2$ and has three fixed points on $F$: the intersection of components $E_1 \cap E_2$, the other fixed point $q_1$ on $E_1$ and the other fixed point $q_2$ on~$E_2$.

In the neighbourhood of $q_1$ the element $g$ acts on $E_1$ as a multiplication by~$\xi_k^a$ for some $a$. In the proof of Lemma \ref{eveninvariantfibre} it was shown that $k = 2a + 1$ and $g$ acts in the neighbourhood of $q_2$ as $\operatorname{diag}\left( \xi_k, \xi_k^a \right)$. Therefore $g$ acts in the neghbourhood of $E_1 \cap E_2$ as $\operatorname{diag}\left( \xi_k^{-a}, \xi_k^{-a} \right)$.

There are three singular points $\pi(q_1)$, $\pi(q_2)$ and $\pi(E_1 \cap E_2)$ on $\pi(F)$. These singularities are toric thus for their minimal resolutions their preimages are chains of rational curves. The selfintersection numbers of these curves for $\pi(q_1)$ and $\pi(q_2)$ are determined by the continued fraction
$$
\frac{2a+1}{a} = 3 - \frac{1}{2 - \frac{1}{\ddots - \frac{1}{2}}}.
$$
\noindent The transform of $\pi(E_1 \cap E_2)$ is a rational $(-2a-1)$-curve.

Let us consider a minimal resolution
$$
\mu: \widetilde{X / G} \rightarrow X / G.
$$
\noindent The preimage $\mu^{-1} \left( \pi(F) \right)$ is a chain of rational curves with the following selfintersections: $-3$, $-2$, $\ldots$, $-2$, $-1$, $-2a-1$, $-1$, $-2$, $\ldots$, $-2$, $-3$. Applying the relative minimal model program one can easily check that $f(F)$ is a singular fibre.

\end{proof}

\begin{proposition}
\label{basequotient}
Let $X$ be a relatively $G$-minimal conic bundle with~$n$ singular fibres and $G_F$ be trivial. Then any MMP-reduction~$Y$ of $X / G$ admits a conic bundle structure. For the number $n$ of singular fibres on $X$ from Table \ref{Table1} the number of singular fibres on $Y$ is no greater than $m$ in Table~\ref{Table1}.

\begin{table}
\caption{}\label{Table1}
\begin{tabular}{|c|c|c|c|}
\hline
{\bf Group} & $n$ & $m$ & {\bf Conditions} \\
\hline
$\CG_{2k}$ & $2kb$ & $a + b$ & $a \leqslant 2$ \\
\hline
$\CG_{2k+1}$ & $a + (2k+1)b$ & $a + b$ & $a \leqslant 2$ \\
\hline
$\DG_{4k}$ & $4kc$ & $a + b + c$ & $a \leqslant 1$, $b \leqslant 2$ \\
\hline
$\DG_{4k+2}$ & $2a + (4k + 2)c$ & $a + b + c$ & $a \leqslant 1$, $b \leqslant 2$ \\
\hline
$\AG_4$ & $4a + 12c$ & $a + b + c$ & $a \leqslant 2$, $b \leqslant 1$ \\
\hline
$\SG_4$ & $8a + 24d$ & $a + b + c + d$ & $a \leqslant 1$, $b \leqslant 1$, $c \leqslant 1$ \\
\hline
$\AG_5$ & $12a + 20b + 60d$ & $a + b + c + d$ & $a \leqslant 1$, $b \leqslant 1$, $c \leqslant 1$ \\
\hline
\end{tabular}
\end{table}

\end{proposition}
\begin{proof}
Denote the map $X \dashrightarrow Y$ by $f$. Let $F$ be a fibre of $\overline{Y} \rightarrow \overline{B} / G$ then $f^{-1}(F)$ is an orbit of $G$. We denote this orbit by $\Omega$ and stabilizer of a fibre in this orbit by $G_S$. By Lemmas \ref{eveninvariantfibre} and \ref{quotientoffibre} if $F$ is a singular fibre on $Y$ then:
\begin{itemize}
\item either $\Omega$ consists of singular fibres and $|G_S|$ is odd.
\item or $\Omega$ consists of smooth fibres and $|G_S|$ is even.
\end{itemize}

Each orbit of fibres corresponds to an orbit of points on the base. Therefore we know numbers of orbits of fibres with non-trivial stabilizers from Lemma \ref{orbits}. Applying this, for each singular fibre $F$ we consider the orbit $\Omega$. The numbers $a$, $b$, $c$ and $d$ in Table \ref{Table1} are numbers of orbits with given length.

The proof is the same for all cases listed in Table \ref{Table1}. Therefore we consider just one of them, for example, $G \cong \DG_{4k+2}$. For any singular fibre $F$ on $Y$ the orbit $\Omega$ consists of $2$ singular fibres, $2k + 1$ smooth fibres or $4k + 2$ singular fibres. If we denote numbers of such fibres $F$ by $a$, $b$ and $c$ respectively then there are $a + b + c$ singular fibres on $Y$ and at least $2a + (4k + 2)c$ singular fibres on $X$.

\end{proof}

Now we prove Theorem \ref{Cbundle}.

\begin{proof}[Proof of Theorem \ref{Cbundle}]
Let $n$ and $m$ be numbers of singular fibres on $\XX \rightarrow \overline{B}$ and $\overline{Y} \rightarrow \overline{B} / G_B$ respectively.

Let us consider a relative $G_B$-MMP-reduction~$Z$ of $X / G_F$. By Lemma \ref{fibrequotient} the number of singular fibres of $\overline{Z} \rightarrow \overline{B}$ is not greater than $n$. The surface $Y$ is a relative MMP-reduction of $Z / G_B$. Applying Proposition $\ref{basequotient}$ to different possibilities for $G_B$ one can check the following:

\begin{itemize}
\item If $n \leqslant 3$ then $m \leqslant 3$.
\item If $n > 3$ then $m \leqslant n$.
\item If $m = n > 3$ then $m = n = 4$ and $G_B \cong \CG_2$ or $G_B \cong \DG_4$.
\end{itemize}

Applying equalities $K_X^2 = 8 - n$ and $K_Y^2 = 8 - m$ finishes the proof.

\end{proof}

Now we prove Proposition \ref{RCbundle}.

\begin{proof}[Proof of Proposition \ref{RCbundle}]
For any relative MMP-reduction $Y$ of $X / G$ one has $K_Y^2 \geqslant 5$ by Theorem \ref{Cbundle}. Thus $X / G \cong Y$ is $\ka$-rational by Corollary \ref{Cbundleratcrit}.
\end{proof}

\section{The Galois group}

In this section we prove Theorem \ref{Cbundleclasses}, Theorem \ref{Oddquotient} and the following proposition:

\begin{proposition}
\label{orbitpossibility}
Let $X \rightarrow B \cong \Pro^1_{\ka}$ be a relatively $G$-minimal conic bundle, such that the group $G$ faithfully acts on the base $B$. Let $\Gamma$ be the Galois group $\Gal\left(\kka / \ka \right)$. Let $Y$ be a relative MMP-reduction over $B / G$ of the minimal resolution of $X / G$ and $f: X \dashrightarrow Y$ be the corresponding map. Let $F$ be a singular fibre over a point $p \in \overline {B}$ such that $f(F)$ is a singular fibre on $\overline{Y}$, $E_1$ and $E_2$ are irreducible components of $F$.

Then there exist elements $g \in G$ and $\gamma \in \Gamma$ such that $g\gamma E_1 = E_2$ and $g\gamma E_2 = E_1$ and (up to conjugation) one of the following possibilities holds:

\begin{enumerate}

\item There exists an element $\delta \in \Gamma$ such that $\delta E_1 = E_2$.

\item The stabilizer of $p$ in $G$ is trivial, $\ord g$ is even and $\gamma p = g^{-1}p$.

\item $G \cong \DG_{4k+2}$, $p$ is a point fixed by the normal cyclic subgroup $\CG_{2k+1}$, $g$ is any element of order $2$ and $\gamma p = gp$.

\item $G \cong \SG_4$, $p$ is a $(123)$-fixed point, $g = (12)$ and $\gamma p = (12)p$.

\item $G \cong \AG_5$, $p$ is a $(12345)$-fixed point, $g = (25)(34)$ and \mbox{$\gamma p = (25)(34)p$}.

\item $G \cong \AG_5$, $p$ is a $(123)$-fixed point, $g = (12)(45)$ and \mbox{$\gamma p = (12)(45)p$}.

\end{enumerate}

\end{proposition}

We start with several auxiliary lemmas.

\begin{lemma}
\label{oddorbits}
In the notation of Proposition \ref{orbitpossibility} if $\ord g$ is odd then there exists an element $\delta \in \Gamma$ such that $\delta E_1 = E_2$.
\end{lemma}

\begin{proof}
One has $g\gamma E_1 = E_2$ and $\ord g$ is odd, therefore
$$
E_2 = \left(g\gamma\right)^{\ord g} E_1 = \gamma^{\ord g} g^{\ord g} E_1 = \gamma^{\ord g} E_1
$$
and $\delta = \gamma^{\ord g}$ is an element of $\Gamma$ such that $\delta E_1 = E_2$.
\end{proof}

\begin{lemma}
\label{stabilizer}
In the notation of Proposition \ref{orbitpossibility} all points in $\langle g \rangle$-orbit has the same stabilizer in $G$.
\end{lemma}
\begin{proof}
Let $H \subset G$ be the stabilizer of the point $p$. For any $h \in H$ one has
$$
hg^l p = \gamma^{-l}h\left(g\gamma\right)^l p = \gamma^{-l} p = g^l \left(g\gamma\right)^{-l} p = g^l p.
$$
Therefore all points in the $\langle g \rangle$-orbit of $p$ are fixed by the group $H$.
\end{proof}

Now we prove Proposition \ref{orbitpossibility}.

\begin{proof}[Proof of Proposition \ref{orbitpossibility}]

The fibre $F$ is singular so its components are permutted by an element
$$
g\gamma \in G \times \Gamma, \qquad g \in G, \qquad \gamma \in \Gamma
$$
\noindent such that $g\gamma E_1 = E_2$, $g\gamma E_2 = E_1$.

By Lemma \ref{oddorbits} if $\ord g$ is odd then there exists an element $\delta \in \Gamma$ such that $\delta E_1 = E_2$. This is case (1) of Proposition \ref{orbitpossibility}.

If $\ord g$ is even and the stabilizer of $p$ in $G$ is trivial then this is case (2) of Proposition \ref{orbitpossibility}.

Assume that $\ord g$ is even and the stabilizer of $p$ in $G$ is non-trivial. By Lemma \ref{stabilizer} any element $h$ of the stabilizer of the point $p$ also fixes the points $gp$ and $g^2p$. If $\ord g > 2$ then $h$ fixes more than $2$ points on $B$, so $h$ acts trivially on $B$. It contradicts the assumption that $G$ acts faithfully on $B$. Thus $\ord g = 2$.

By Lemma $\ref{eveninvariantfibre}$ the stabilizer of the point $p$ is generated by an element $h \in G$ of odd order. By Lemma \ref{stabilizer} one has $hgp = gp$ so by Lemma~\ref{fixedpoints} the points $p$ and $gp$ are fixed only by the element $h$ and its powers. Thus one has $ghg^{-1} \in \langle h \rangle$. Applying Lemma \ref{orbits} one can consider all possibilities (up to conjugation) for $G$, $h$ and $g$ and check that they all are listed in cases (3)--(6) of Proposition \ref{orbitpossibility}:

\begin{itemize}

\item $G \cong \DG_{4k+2}$, $h$ is a generator of the normal subgroup $\CG_{2k + 1}$ and $g$ is any element of order $2$.

\item $G \cong \SG_4$, $h = (123)$ and $g = (12)$.

\item $G \cong \AG_5$, $h = (12345)$ and $g = (25)(34)$.

\item $G \cong \AG_5$, $h = (123)$ and $g = (12)(45)$.

\end{itemize}

\end{proof}

In Section 6 we will show that all possibilities listed in Proposition~\ref{orbitpossibility} occur for a suitable field $\ka$. Now we show that some quotients of $\ka$-rational conic bundles by groups faithfully acting on the base are $\ka$-birationally equivalent.

\begin{lemma}
\label{oddconic}
Let a group $G = \CG_k$ or $G = \DG_{2k}$ act on a relatively \mbox{$G$-minimal} $\ka$-rational conic bundle $X$. Let $N$ be a maximal cyclic subgroup in $G$ of odd order. Then $X / N$ is $\ka$-rational.
\end{lemma}
\begin{proof}
Let $\Gamma = \Gal\left(\kka / \ka \right)$. Let $F$ be a singular fibre over a point $p \in \overline{B}$. By Lemma \ref{oddorbits} if for an element $g \in N$ there exists an element $\gamma \in \Gamma$ such that $g \gamma$ permutes components of $F$, then there exists an element $\delta \in \Gamma$ permuting the components of a singular fibre $F$. Otherwise, there is no element $\varepsilon \in \Gamma$ permuting the components of the image of $F$ in $X / N$. Therefore by Lemma \ref{quotientoffibre} the number of singular fibres on relative $G / N$-MMP reduction $Y$ of $X / N$ whose components are permuted by an element of $\Gamma$ is not greater than the number of singular fibres on $X$ whose components are permuted by an element of $\Gamma$.

Let $\widetilde{X}$ and $\widetilde{Y}$ be relative minimal models of $X$ and $Y$ respectively. Then the number of singular fibres of $\widetilde{Y}$ is not greater than the number of singular fibres of $\widetilde{X}$. Thus $K_{\widetilde{Y}}^2 \geqslant K_{\widetilde{X}}^2 \geqslant 5$. Therefore $Y \approx \widetilde{Y}$ are \mbox{$\ka$-rational} by Corollary \ref{Cbundleratcrit}.
\end{proof}

Now we prove Theorem \ref{Cbundleclasses}.

\begin{proof}[Proof of Theorem \ref{Cbundleclasses}]
Let us consider a $\ka$-rational $G$-surface $X$. Obviously, the quotient $X / G$ is birationally equivalent to $\widetilde{X} / G$ where $\widetilde{X}$ is a $G$-equivariant minimal model of $X$. By Theorem \ref{Minclass} the surface $\widetilde{X}$ is a del Pezzo surface or admits a structure of $G$-equivariant conic bundle.

If $\widetilde{X}$ is a del Pezzo surface then we are done.

If $\widetilde{X} \rightarrow B \cong \Pro^1_{\ka}$ is a conic bundle then there is a normal subgroup $G_F \subset G$ acting trivially on $B$. The quotient $\widetilde{X} / G_F$ is $\ka$-rational by Lemma \ref{fibrequotient} and $X / G$ is $\ka$-birationally equivalent to $(\widetilde{X / G_F}) /G_B$, where $G_B = G / G_F$ and $\widetilde{X / G_F}$ is the minimal resolution of singularities of~$X / G_F$.

Let $Y$ be a relative $G_B$-MMP reduction of $\widetilde{X / G_F}$. Then the group $G_B$ faithfully acts on the base $B$ thus $G_B$ is isomorphic to $\CG_k$, $\DG_{2k}$, $\AG_4$, $\SG_4$ or~$\AG_5$. In the first two cases we can consider a maximal cyclic group $N$ of odd order. By Lemma \ref{oddconic} the quotient $Y / N$ is a $\ka$-rational conic bundle. Thus the quotient $X / G$ is $\ka$-birationally equivalent to a quotient of the $\ka$-rational conic bundle $Y / N$ by a cyclic or dihedral group $G / N$ of order~$2^n$.
\end{proof}

To prove Theorem \ref{Oddquotient} we use the following well-known lemma.

\begin{lemma}
\label{toric}
Let $\XX$ be an $n$-dimensional toric variety over a field $\kka$ of arbitrary characteristic and let $G$ be a finite subgroup conjugate to a subgroup of an $n$-dimensional torus $\overline{\mathbb{T}}^n \subset \XX$ acting on $\XX$. Then the quotient $\XX / G$ is a toric variety.
\end{lemma}

Now we prove Theorem \ref{Oddquotient}.

\begin{proof}[Proof of Theorem \ref{Oddquotient}]

Consider a $G$-minimal model $X$ of $S$. Then $X$ is either a conic bundle or a del Pezzo surface by Theorem \ref{Minclass}.

If $X$ is a conic bundle then $X / G$ is $\ka$-rational by Theorem \ref{Cbundleclasses}. So we can assume that $X$ is a del Pezzo surface.

The classification of finite groups of automorphisms of del Pezzo surfaces is well-known (see \cite{DI1} and \cite[footnote to Theorem~1.1]{Pr13}). If $X$ is a del Pezzo surface of degree $5$ or less and $G$ is a group of odd order, $|G| > 9$ and $|G| \ne 15$ then one can check that order of $G$ is divisible by $3$ and $G$ is not cyclic. It contradicts assumptions of the corollary.

If $X$ is a $G$-minimal del Pezzo surface of degree $6$ or more then $\XX$ is isomorphic to $\Pro^2_{\kka}$, $\Pro^1_{\kka} \times \Pro^1_{\kka}$ or to a blowup of $\Pro^2_{\kka}$ at three points not lying on a line. The corresponding groups of automorphisms are $\mathrm{PGL}_3 \left( \kka \right)$, $\left( \mathrm{PGL}_2 \left( \kka \right) \times \mathrm{PGL}_2 \left( \kka \right) \right) \rtimes \CG_2$ and $\overline{\mathbb{T}}^2 \rtimes \SG_3$ respectively.

Any cyclic subgroup of $\mathrm{PGL}_3 \left( \kka \right)$ is conjugate to a diagonal subgroup. From classification of finite subgroups of $\mathrm{PGL}_3 \left( \kka \right)$ (see \cite[Chapter V]{Bl17}) one can easily see that if $|G|$ is odd and is not divisible by $3$ then $G$ is also conjugate to a diagonal subgroup. Any diagonal subgroup is a subgroup of a torus acting on $\Pro^2_{\kka}$.

If $\XX \cong \Pro^1_{\kka} \times \Pro^1_{\kka}$ and $|G|$ is odd then $G$ is a subgroup of~\mbox{$\mathrm{PGL}_2 \left( \kka \right) \times \mathrm{PGL}_2 \left( \kka \right)$}. By Theorem \ref{PGL2class} the group $G$ is a subgroup of $\CG_m \times \CG_l$ for some odd $m$ and $l$. Thus it is a subgroup of a torus acting on $\XX$.

If $\XX \rightarrow \Pro^2_{\kka}$ is a blowup of three points not lying on a line and $G$ is not a subgroup of a torus acting on $\XX$ then the composition of maps
$$
G \hookrightarrow \overline{\mathbb{T}}^2 \rtimes \SG_3 \twoheadrightarrow \SG_3
$$
\noindent has a nontrivial image $\CG_3$. In particular, $|G|$ is divisible by $3$. Therefore by conditions of Theorem \ref{Oddquotient} the group $G$ should be cyclic. The map $\XX \rightarrow \Pro^2_{\kka}$ is $G$-equivariant. Thus $G$ is a subgroup of a torus acting on $\Pro^2_{\kka}$ and $\XX$.

If $G$ is a subgroup of a torus acting on $\XX$ then by Lemma \ref{toric} the quotient $X / G$ is a $\ka$-form of a toric surface. Thus an MMP-reduction $Y$ of $X$ is a $\ka$-form of toric surface which is either a del Pezzo surface of degree no less than $6$ or a conic bundle with no more than $2$ singular fibres since other minimal surfaces are not toric. Therefore $K_Y^2 \geqslant 6$. So both $X / G \approx Y$ and $Y$ are $\ka$-rational by Theorem \ref{ratcrit}.
\end{proof}

\section{Example of nonrationality}

In this section we construct explicit examples of non-$\ka$-rational quotient of $\ka$-rational conic bundle. Note that by Theorem \ref{Cbundle} and Corollary \ref{Cbundleratcrit} to construct such an example we should realise one of the cases (2)--(6) of Proposition \ref{orbitpossibility}. We show that all these cases can be realized for suitable field $\ka$.

\begin{example}
\label{keyexample}

Consider a projective plane $\Pro^2_{\ka}$ with coordinates $x$, $y$ and $z$ and a projective line $\Pro^1_{\ka}$ with coordinates $t_1$ and $t_0$. Let a finite group $G$ act on $\Pro^1_{\ka}$ and $g \in G$ be an element of even order $l$ such that points $(0 : 1)$ and $(1 : 0)$ are fixed by $g$. In particular, one has
$$
g \left(\lambda : 1 \right) = \left( \xi_l \lambda : 1 \right).
$$
The group $G$ naturally acts on homogeneous polynomials in variables $t_1$ and $t_0$.

Let $Z$ be a hypersurface in $\Pro^2_{\ka} \times \Pro^1_{\ka}$ given by the equation

$$
Ax^2 \prod \limits_{i=1}^n \prod \limits_{h \in G} h \left( t_1 - \lambda_i t_0 \right) + By^2 \left( \prod \limits_{h \in G} h \left( t_0 \right) \right)^n + Cz^2 \left( \prod \limits_{h \in G} h \left( t_0 \right) \right)^n = 0.
$$

The projection $\pi: \Pro^2_{\ka} \times \Pro^1_{\ka} \rightarrow \Pro^1_{\ka}$ defines on $Z$ a structure of a bundle over $B \cong \Pro^1_{\ka}$ whose general fibre is a smooth conic. Let $X \rightarrow B$ be a relative $G$-MMP-reduction over $B$ of $Z$. Note that the conic bundle $\pi: Z \rightarrow B$ has singular fibres exactly over points $h\left( \lambda_i : 1 \right)$ and $h(1 : 0)$, $h \in G$. Moreover, by Lemma~\ref{eveninvariantfibre} the fibres of $\varphi: X \rightarrow B$ over points $h(1 : 0)$ are smooth since these points are fixed by the elements of even order $hgh^{-1}$.

Now assume that in the field $\ka$ there is an element $u$ such that
$$
\Gal \left( \ka \left( \sqrt[l]{u} \right) / \ka \right) \cong \CG_l.
$$
\noindent That means that the equation $x^l = u$ has no solutions in $\ka$ and $\xi_l \in \ka$. Let $\lambda_i = \mu_i \sqrt[l]{u}$, $\mu_i \in \ka$ and $\frac{B}{C} = -u$. In this case the surface $Z$ is well defined since $G$-orbit of the point $\left( \lambda_i : 1 \right)$ includes all solutions of the equation $t_1^l - \mu_i^l u t_0^l = 0$ defined over $\ka$. Let
$$
\gamma \in \Gal \left( \ka \left( \sqrt[l]{u} \right) / \ka \right)
$$
\noindent be the generator of the group acting by $\gamma: \lambda_i \mapsto \xi^{-1}_l \lambda_i$.

Now fix a $\ka$-point $q \in \Pro^1_{\ka}$ and find a coefficent $A$ such that fibre over $q$ contains a $\ka$-point.

We show that this construction gives examples of $\ka$-nonrational quotients of $\ka$-rational surfaces.

\end{example}

We want to show that this example exists for some field $\ka$ and group $G$ and gives a class of quotients of \mbox{$\ka$-rational} surfaces which is \mbox{$\ka$-birationally} unbounded class.

\begin{proposition}
\label{ratex}

Let a field $\ka$ contain an element $u$ such that
$$
\Gal \left( \ka \left( \sqrt[l]{u} \right) / \ka \right) \cong \CG_l.
$$
In the notation of Example \ref{keyexample} assume that the points $(\lambda_i : 1)$ have trivial stabilizers in $G$. Then the surface $X$ is $\ka$-rational and any relative MMP-reduction $Y$ of $X / G$ admits a conic bundle structure with at least $n$ singular fibres. In particular, if $n > 3$ then $Y$ is not $\ka$-rational. If $n \geqslant 8$ then the construction of Examle \ref{keyexample} gives an $(n - 3)$-dimensional family of surfaces of different $\ka$-birational types.

\end{proposition}

\begin{proof}

All components of singular fibres of $\varphi: X \rightarrow B$ are defined over \mbox{$\ka \left( \sqrt[l]{u} \right)$} and for any singular fibre $F$ of $\varphi$ one has $\gamma F \ne F$. So one can $\Gal \left( \ka \left( \sqrt[l]{u} \right) / \ka \right)$-equivariantly contract a component in each singular fibre of $\varphi$ and obtain a conic bundle $S \rightarrow B$ without singular fibres. By construction $S$ has a $\ka$-point. So both $S$ and $X \approx S$ are $\ka$-rational by Corollary~\ref{Cbundleratcrit}.

Consider a singular fibre $F$ of $\varphi$ over $\left( \lambda_i : 1 \right)$. Let $f: X \rightarrow X / G$ be the quotient morphism. Note that the image of the point $\left( \lambda_i : 1 \right)$ is a $\ka$-point on $\Pro^1_{\ka} / G$ since the set $\left\{ h\left( \lambda_i : 1 \right) \right\}_{h \in G}$ is defined over $\ka$. The image $f(F)$ is a singular conic given by equation $By^2 + Cz^2 = 0$. Therefore the element $\gamma$ permutes components of $f(F)$ since $\frac{B}{C} = -u$. Thus the conic bundle $Y$ has at least $n$ singular fibres. If $n > 3$ then $Y$ is not $\ka$-rational by \mbox{Corollary \ref{Cbundleratcrit}}.

The conic bundle $Y \rightarrow B / G$ has $n$ singular fibres over points \mbox{$p_1$, \ldots, $p_n$} which are the images of points $\left( \lambda_i : 1 \right)$ where $\lambda_i = \mu_i \sqrt[l]{u}$. The elements $\mu_i$ can be any elements of $\ka$ such that points $\left( \lambda_i : 1 \right)$ are not fixed by any element of $G$ and lie in different orbits of $G$. Therefore $n$-tuples $\left( \mu_1, \ldots, \mu_n \right)$ form an $n$-dimensional family. The group $G$ is finite therefore we have an $n$-dimensional family of $n$-tuples of points $\left( p_1, \ldots, p_n \right)$ on $B / G$. Since $n \geqslant 8$ by Lemma \ref{rigidbundle} two conic bundles can be birationally equivalent only if the group $\mathrm{PGL}_2\left( \ka \right)$ moves one $n$-tuple of points to other. One has $\operatorname{dim}_{\ka} \mathrm{PGL}_2\left( \ka \right) = 3$ therefore the construction of Example \ref{keyexample} gives an $(n - 3)$-dimensional family of surfaces of different $\ka$-birational types.

\end{proof}

\begin{lemma}
\label{UmboundG}

Let $\ka$ be a field of characteristic zero such that not all elements of $\ka$ are squares, $G$ be a finite group acting on $\Pro^1_{\ka}$ and $g \in G$ be an element of order $2$ such that its fixed points on $\Pro^1_{\ka}$ are defined over $\ka$. Then the class of quotients of $\ka$-rational surfaces by the group $G$ is $\ka$-birationally unbounded.

\end{lemma}

\begin{proof}
Assume that for the group $G$ any $G$-quotient of a $\ka$-rational surface is $\ka$-birationally equivalent to a fibre of some family $\psi: \mathcal{X} \rightarrow \mathcal{S}$. Let $m$ be dimension of $\mathcal{S}$. Put $n = m + 8$.

The fixed points of the element $g \in G$ on $\Pro^1_{\ka}$ are defined over $\ka$. So we can assume that these points are $(1 : 0)$ and $(0 : 1)$. Now we can apply the construction of Example \ref{keyexample} to the group $G$, field $\ka$ and $n$ singular fibres on the quotient surface. By Proposition \ref{ratex} the dimension of family of such surfaces up to $\ka$-birationally equivalence is $n - 3$ that is greater than $m$. The obtained contradiction finishes the proof.

\end{proof}

\begin{remark}
\label{increasing}

In Lemma \ref{UmboundG} we actually show that for an arbitrarily large number $n$ there exists an $(n-3)$-dimensional family of Galois $\ka$-unirational conic bundles with $n$ singular fibres. This method is similar to \cite{Ok10}.

\end{remark}

Now we prove Theorem \ref{Unboundness}.

\begin{proof}[Proof of Theorem \ref{Unboundness}]

By Lemma \ref{UmboundG} to prove Theorem \ref{Unboundness} it is sufficient to find a representation of $G$ such that the fixed point of an element of order $2$ in $G$ are defined over $\ka$. The required representations are given in Lemmas \ref{Crepr}, \ref{Drepr}, \ref{A4repr}, \ref{S4repr} and \ref{A5repr}.

\end{proof}

\begin{remark}
\label{nonratquot}

Let us consider Example \ref{keyexample} and put some $\lambda_i$ being elements of $\ka$. Then the conic bundle $X \rightarrow B$ has some singular fibres. If the number of such fibres is greater than $3$ then $X$ is not $\ka$-rational. The family of conic bundles with fixed number of singular fibres is $\ka$-birationally bounded, but similarly to Theorem \ref{Unboundness} we can show that the family of $G$-quotients of such surfaces is $\ka$-birationally unbounded. Therefore there exists a non-$\ka$-rational $\ka$-birational type of surfaces which $G$-quotients are $\ka$-birationally unbounded.

\end{remark}

\begin{proof}[Proof of Corollary \ref{Uncremona}]
Let $G_1$ and $G_2$ be two subgroups of $\operatorname{Bir} \Pro^2_{\ka}$. Consider $\ka$-rational surfaces $X_1$ and $X_2$ such that the actions of $G_1$ on $X_1$ and $G_2$ on $X_2$ are regular. Recall that $G_1$ and $G_2$ are conjugate in $\operatorname{Bir} \Pro^2_{\ka}$ if and only if there exist a surface $X$ and two equivariant birational morhpisms $\pi_1: X \rightarrow X_1$ and $\pi_2: X \rightarrow X_2$ such that the actions of $G_1$ and $G_2$ coincide on $X$.

Therefore if $G_1$ and $G_2$ are conjugate then
$$
X_1 / G_1 \approx X / G_1 \approx X / G_2 \approx X_2 / G_2.
$$
Thus if for two subgroups $G_1 \cong G_2$ in $\operatorname{Bir} \Pro^2_{\ka}$ the quotients $X_1 / G_1$ and $X_2 / G_2$ are not $\ka$-birationally equivalent then $G_1$ and $G_2$ are not conjugate in $\operatorname{Bir} \Pro^2_{\ka}$.

By Theorem \ref{Unboundness} the class of $G$-quotients of $\ka$-rational surfaces is $\ka$-birationally unbounded. Therefore there are infinitely many subgroups isomorphic to $G$ in $\operatorname{Bir} \Pro^2_{\ka}$ up to conjugation.

\end{proof}

In Proposition \ref{ratex} we assume that points $(\lambda_i : 1)$ have trivial stabilizers in $G$. This case corresponds to case (2) of Proposition \ref{orbitpossibility}. Now we show that Example \ref{keyexample} works for orbits consisting of points with non-trivial stabilizer (cases (3)--(6) of Proposition \ref{orbitpossibility}).

\begin{lemma}
\label{examplest}

Let $G$ be a finite group acting on $B = \Pro^1_{\ka}$, $g$ be an element of order $2$ such that its fixed points are defined over $\ka$, $h$ be an element of odd order such that its fixed points are not defined over $\ka$ and \mbox{$ghg^{-1} \in \langle h \rangle$}. Let $p$ be a fixed point of $h$ on $\overline{B}$ and $q$ be its image on $\overline{B} / G$.

Then there exists a relatively $G$-minimal conic bundle $X \rightarrow B$ such that the fibre of $\XX \rightarrow \overline{B}$ over $p$ is singular and for a relative MMP-reduction $Y \rightarrow B /G$ of $X / G \rightarrow B / G$ the fibre of $\overline{Y} \rightarrow \overline{B} / G$ over $q$ is singular.

\end{lemma}

\begin{proof}

The fixed points of $g$ are defined over $\ka$ so we can assume that $g$ acts as
$$
\left( t_1 : t_0 \right) \mapsto \left( -t_1 : t_0 \right)
$$
\noindent and its fixed points are $(1 : 0)$ and $(0 : 1)$. Thus we can apply construction of Example \ref{keyexample}. One has $ghg^{-1} \in \langle h \rangle$ therefore the element $g$ permutes fixed points of $h$. So the element $h$ has fixed points $(\lambda : 1)$, $(-\lambda : 1)$. These points are permuted by the group
$$
\CG_2 \cong \Gal \left( \ka \left( \lambda \right) / \ka \right),
$$
\noindent thus $\lambda^2 \in \ka$. In the notation of Example \ref{keyexample} we can put $u = \lambda^2$, $\lambda_1 = \lambda$.

The fibre of $\XX \rightarrow \overline{B}$ over $p = (\lambda : 1)$ is singular since the element $g \gamma$ permutes components of the fibre of $\overline{Z} \rightarrow \overline{B}$ over $p$ (see the notation of Example \ref{keyexample}). The fibre of $Y \rightarrow B / G$ over $q$ is singular by Lemma~\ref{quotientofoddfibre}.

\end{proof}

\begin{proposition}

In the notation of Proposition \ref{orbitpossibility} for each case (3)--(6) there exists a field $\ka$ and a $G$-equivariant conic bundle $X \rightarrow B$ such that the components of $F$ are not permuted by $\Gamma$ and the components of $f(F)$ are permuted by~$\Gamma$.

In particular, this condition holds

\begin{itemize}

\item for the case (3) if $\cos \frac{2\pi}{2k+1} \in \ka$ and $\xi_{2k + 1} \notin \ka$;

\item for the case (4) if $i\sqrt{2} \in \ka$ and $\xi_3 \notin \ka$;

\item for the case (5) if $i \in \ka$, $\sqrt{5} \in \ka$ and $\xi_5 \notin \ka$;

\item for the case (6) if $i \in \ka$, $\sqrt{5} \in \ka$ and $\xi_3 \notin \ka$.

\end{itemize}

\end{proposition}

\begin{proof}

In each group $\DG_{4k + 2}$, $\SG_4$ and $\AG_5$ there exist two elements $g$ and $h$ such that $\ord g = 2$, $\ord h$ is odd and $ghg^{-1} \in \langle h \rangle$. By Lemma \ref{examplest} it is sufficient to find a field $\ka$ and a representation of the group $G$ in $\mathrm{PGL}_2(\ka)$ such that the fixed points of the element $g$ are $\ka$-points and the fixed points of an element $h$ are not defined over $\ka$.

By Lemma \ref{Fixeddefined} if the fixed points of element $h$ are not defined over $\ka$ then $\xi_{\ord h} \notin \ka$. The representations given in Lemmas \ref{Drepr}, \ref{S4repr} and \ref{A5repr} satisfy the previous conditions.

\end{proof}

\bibliographystyle{alpha}
\bibliography{my_ref}
\end{document}